\documentclass[11pt,reqno]{amsart}

\usepackage{amsmath,amsthm,amscd,amsfonts,amssymb,color}
\usepackage{cite}
\usepackage[mathscr]{eucal}
\usepackage[bookmarksnumbered,colorlinks,plainpages]{hyperref}
\newtheorem{theorem}{Theorem}[section]

\newtheorem*{Acknowledgement}{\textnormal{\textbf{Acknowledgement}}}

\newtheorem{proposition}[theorem]{Proposition}
\newtheorem{corollary}[theorem]{Corollary}

\theoremstyle{definition}
\newtheorem{definition}[theorem]{Definition}

\newtheorem{Open Prob}[theorem]{Open Problem}
\theoremstyle{remark}
\newtheorem{remark}[theorem]{Remark}

\numberwithin{equation}{section}
\def\DJ{\leavevmode\setbox0=\hbox{D}\kern0pt\rlap{\kern.04em\raise.188\ht0\hbox{-}}D}
\begin{document}

\title[Some remarks on the metrizability of some metric-like structures]{Some remarks on the metrizability of some well known generalized metric-like structures}

\author[S.\ Som,  A. Petru\c{s}el, L.K.\ Dey]
{Sumit Som$^{1}$, Adrian Petru\c{s}el$^{2}$, Lakshmi Kanta Dey$^{3}$}

\address{{$^{1}$} Sumit Som,
                    Department of Mathematics,
                    National Institute of Technology
                    Durgapur, India.}
                    \email{somkakdwip@gmail.com}
\address{{$^{2}$} Adrian Petru\c{s}el,
                    Department of Mathematics,
                    Babe\c{s}-Bolyai University Cluj-Napoca, Romania and
                    Academy of Romanian Scientists Bucharest, Romania}
                    \email{petrusel@math.ubbcluj.ro} 
\address{{$^{3}$} Lakshmi Kanta Dey,
                    Department of Mathematics,
                    National Institute of Technology
                    Durgapur, India.}
                    \email{lakshmikdey@yahoo.co.in}

\keywords{ $b$-metric space, $\mathcal{F}$-metric space, $\theta$-metric space, metrizability.\\
\indent 2010 {\it Mathematics Subject Classification}. $54$E$35$, $54$H$99$.}

\begin{abstract}
In \cite[\, An, V.T., Tuyen, Q.L. and Dung, V.N., Stone-type theorem on $b$-metric spaces and applications, Topology Appl. 185-186 (2015), 50-64.]{an}, An et al. had provided a sufficient condition for $b$-metric spaces to be metrizable. However, their proof of metrizability relied on an assumption that the distance function is continuous in one variable. In this short note, we improve upon this result in a more simplified way without considering any assumption on the distance function. Moreover, we provide two shorter proofs of the metrizability of $\mathcal{F}$-metric spaces recently introduced by Jleli and Samet in \cite[\, Jleli, M. and Samet, B., On a new generalization of metric spaces, J. Fixed Point Theory Appl. (2018) 20:128]{JS1}. Lastly, in this short note, we give an alternative proof of the metrizability of $\theta$-metric spaces introduced by Khojasteh et al. in \cite[\, Khojasteh, F., Karapinar, E. and Radenovic, S., $\theta$-metric space: A Generalization, Math. Probl. Eng. Volume 2013, Article 504609, 7 pages]{ks}.

\end{abstract}

\maketitle

\setcounter{page}{1}

\section{\bf Metrizability of $b$-metric spaces}
\baselineskip .55 cm
In the year 1993, Czerwik \cite{cz1} had introduced the notion of a $b$-metric as a generalization of a metric and further, in 1998, Czerwik \cite{cz2} had modified this notion where the coefficient $2$ was replaced by coefficient $K\geq 1.$ Surprisingly in the year 1998, Aimar et al. \cite{ai} proved the metrizability of such spaces. In this sequel, intendant readers can see \cite{C} for some more results.  In the year 2010, Khamsi and Hussain \cite{kh} defined the concept of a $b$-metric under the name metric-type spaces where they had considered the coefficient to be $K>0.$ To avoid confusion, the metric-type in the sense of Khamsi and Hussain \cite{kh} will be called $b$-metric in this short note. Before going further, we like to recall the definition of a $b$-metric space from \cite{kh} as follows:

\begin{definition} \cite[\, Definition 6.]{kh}
Let $X$ be a non-empty set and $K>0.$ A distance function $D: X \times X\rightarrow [0,\infty)$ is said to be a $b$-metric on $X$ if it satisfies the following conditions:
\begin{enumerate}
\item[(i)] $D(x,y)=0\Longleftrightarrow x=y~~\mbox{for all}~(x,y)\in X \times X$;

\item[(ii)] $D(x,y)=D(y,x)~~\mbox{for all}~(x,y)\in X \times X$;

\item[(iii)]$D(x,z)\leq K[D(x,y)+D(y,z)]~~\mbox{for all}~ x,y,z\in X$.
\end{enumerate}
\end{definition}
Then the triple $(X,D,K)$ is called a $b$-metric space. If we take $K=1,$ then $X$  becomes a metric space. So $b$-metric spaces are more general than the standard metric spaces. Again, in the year $2015$,   An et al.  \cite{an} presented a proof for the metrizability of $b$-metric spaces with coefficient $K>0$. However, they proved the metrizability result on an assumption that the distance function is continuous in one variable. We will state first the main theorem and its corollary due to An et al. 

\begin{theorem} \cite[\, Theorem 3.15.]{an} \label{con1}
Let $(X,D,K)$ be a $b$-metric space. If $D$ is continuous in one variable then every open cover of $X$ has an open refinement which is both locally finite and $\sigma$-discrete.
\end{theorem}

\begin{corollary} \cite[\, Corollary 3.17.]{an} \label{con2}
Let $(X,D,K)$ be a $b$-metric space. If $D$ is continuous in one variable then $X$ is metrizable.
\end{corollary}

One of the main motivation of this short note is to give a simple proof of the metrizability of $b$-metric spaces with coefficient $K>0$ without considering any assumptions. We use metrization theorem due to Niemytski and Wilson in our proof. Before proceeding to our metrizability result, we like to recall the metrization theorem due to Niemytski and Wilson as follows: 

\begin{theorem} \cite[\, Page 137.]{FR} \label{ew}
Let $X$ be a topological space and $F:X \times X\rightarrow [0,\infty)$ be a distance function on $X$. If the distance function $F$ satisfies
\begin{enumerate}
\item[(i)] $F(x,y)=0\Longleftrightarrow x=y~~\mbox{for all}~(x,y)\in X \times X$;

\item[(ii)] $F(x,y)=F(y,x)~~\mbox{for all}~(x,y)\in X \times X$
\end{enumerate}

and one of the following conditions:

\begin{enumerate}
\item[(iii-A)] Given a point $a \in X$ and a number $\varepsilon>0$, there exists $\phi(a,\varepsilon)>0$ such that if $F(a,b)<\phi(a,\varepsilon)$ and $F(b,c)<\phi(a,\varepsilon)$ then $F(a,c)<\varepsilon$;

\item[(iii-B)] if $a\in X$ and $\{a_n\}_{n\in \mathbb{N}}, \{b_n\}_{n\in \mathbb{N}}$ are two sequences in $X$ such that $F(a_n,a)\rightarrow 0$ and  $F(a_n,b_n)\rightarrow 0$ as $n \rightarrow \infty$ then $F(b_n,a)\rightarrow 0$ as $n \rightarrow \infty$;

\item[(iii-C)] for each point $a\in X$ and positive number $k,$ there is a positive number $r$ such that if $b\in X$ for which $F(a,b)\geq k,$ and $c$ is any point then $ F(a,c)+F(b,c)\geq r$,
\end{enumerate}
then the topological space $X$ is metrizable. 
\end{theorem}

Niemytski and Wilson showed that the three conditions (iii-A), (iii-B), (iii-C) are equivalent. Any distance function which satisfies any one of the three conditions, is called locally regular. Now in the upcoming theorem we present a shorter proof of the metrizability of $b$-metric spaces.

\begin{theorem}\label{FMSS}
Let $(X,D,K),~K>0$ be a $b$-metric space. Then $X$ is metrizable.
\end{theorem}

\begin{proof}
 Let $(X,D,K)$ be a $b$-metric space. By the definition of a $b$-metric space, the distance function $D:X \times X\rightarrow [0,\infty)$ on $X$ satisfies the first two conditions of Niemytski and Wilson's metrization result, i.e,
\begin{enumerate}
\item[(i)] $D(x,y)=0\Longleftrightarrow x=y~~\mbox{for all}~(x,y)\in X \times X$;

\item[(ii)] $D(x,y)=D(y,x)~~\mbox{for all}~(x,y)\in X \times X$.
\end{enumerate}
Now we prove the third condition, i.e., the "locally regular" condition and for that, we prove the condition (iii-C) of Theorem \ref{ew}. Let $a\in X$ and $t$ be a positive real number. Assume that $b \in X$ such that $D(a,b) \geq t.$ If $c$ is any point in $X$ then by the definition of a $b$-metric space we have,

$$D(a,b)\leq K\Big(D(a,c)+ D(c,b)\Big)$$
$$\Longrightarrow \Big(D(a,c)+ D(c,b)\Big) \geq \frac{t}{K}=r>0.$$

This shows that the distance function $D:X \times X\rightarrow [0,\infty)$ of a $b$-metric space satisfies the locally regular condition. Similarly conditions (iii-A) and (iii-B) of Theorem \ref{ew} are easily satisfied by any $b$-metric. Consequently, by  Niemytski and Wilson's metrization theorem we can conclude that the $b$-metric space $X$ is metrizable.
\end{proof}
\begin{remark}Certainly the above metrizability result is superior, in some sense, to the ones in \cite{ai,an,C}. 
\end{remark}

\begin{remark}
From Theorem \ref{FMSS}, we can conclude that if $(X,D,K),~ K>0$ is a $b$-metric space, then there exists a metric $d:X \times X\rightarrow [0,\infty)$ on $X$ such that $X$ is metrizable with respect to the metric $d.$ Thus, the topological properties of $b$-metric spaces discussed in \cite[\, Proposition 2, Proposition 3]{kh} are equivalent to those of the standard metric spaces.
\end{remark}

\section{ \bf Metrizability of $\mathcal{F}$-metric spaces}

Recently, Jleli and Samet \cite{JS1} proposed a new generalization of the usual notion of metric spaces. By means of a certain class of functions, the authors defined the notion of an $\mathcal{F}$-metric space. Let us first recall the definition of such  spaces. Let $\mathcal{F}$ denote the class of functions $f:(0,\infty)\rightarrow \mathbb{R}$ which satisfy the following conditions:

($\mathcal{F}_1$) $f$ is non-decreasing, i.e., $ 0<s<t\Rightarrow f(s)\leq f(t)$.

($\mathcal{F}_2$) For every sequence $\{t_n\}_{n\in \mathbb{N}}\subseteq (0,+\infty)$, we have
$$\lim_{n\to +\infty}t_n=0 \Longleftrightarrow \lim_{n\to +\infty}f(t_n)=-\infty.$$


The definition of an $\mathcal{F}$-metric space has been introduced as follows.

\begin{definition} \cite[\, Definition 2.1.]{JS1} \label{D1}
Let $X$ be a non-empty set and $D:X\times X\rightarrow [0,\infty)$ be a given mapping. Suppose there exists $(f,\alpha)\in \mathcal{F}\times [0,\infty)$ such that:
\begin{enumerate}
\item[(D1)] $D(x,y)=0\Longleftrightarrow x=y~~\mbox{for all}~(x,y)\in X \times X$.
\item[(D2)] $D(x,y)=D(y,x)~~\mbox{for all}~(x,y)\in X \times X$.
\item[(D3)] For every $(x,y)\in X\times X$, for each $N\in \mathbb{N},~ N\geq2$ and for every $(u_i)_{i=1}^{N}\subseteq X $ with $(u_1,u_N)=(x,y)$, we have
$$D(x,y)>0 \Longrightarrow f(D(x,y))\leq f\left(\sum_{i=1}^{N-1}D(u_i,u_{i+1})\right)+ \alpha.$$
\end{enumerate}
 
\end{definition}
Then $D$ is said to be an $\mathcal{F}$-metric on $X$ and the pair $(X,D)$ is said to be an $\mathcal{F}$-metric space. Hence, the class of all $\mathcal{F}$-metric spaces contain the class of all metric spaces for any $f\in \mathcal{F}$ and $\alpha=0.$ The following definitions and propositions from \cite{JS1} will be needed.

\begin{definition} \cite[\, Definition 4.1.]{JS1}
Let $(X,D)$ be an $\mathcal{F}$-metric space. A subset $C$ of $X$ is said to be $\mathcal{F}$-open if for every $x \in C$, there is some $r>0$ such that $B(x,r)\subset C$ where $$B(x,r)=\{y\in X : D(y,x)<r\}.$$ We say that a subset $C$ of $X$ is $\mathcal{F}$-closed if $X\setminus C$ is $\mathcal{F}$-open. The family of all $\mathcal{F}$-open subsets of $X$ is denoted by $\tau_{\mathcal{F}}.$
\end{definition}

\begin{definition} \cite[\, Definition 4.3.]{JS1}
Let $(X,D)$ be an $\mathcal{F}$-metric space. Let $\{x_n\}_{n\in \mathbb{N}}$ be a sequence in $X.$ We say that $\{x_n\}_{n\in \mathbb{N}}$ is $\mathcal{F}$-convergent to $x\in X$ if $\{x_n\}_{n\in \mathbb{N}}$ is convergent to $x\in X$ with respect to the topology $\tau_{\mathcal{F}}.$
\end{definition}

\begin{proposition}  \cite[\, Proposition 4.4.]{JS1} \label{p1}
Let $(X,D)$ be an $\mathcal{F}$-metric space. Then, for any nonempty subset $A$ of $X$, we have $$x\in \bar{A},~ r>0\Longrightarrow B(x,r)\cap A \neq \phi.$$
\end{proposition}

\begin{proposition}  \cite[\, Proposition 4.5.]{JS1} \label{p2}
Let $(X,D)$ be an $\mathcal{F}$-metric space. Let $\{x_n\}_{n\in \mathbb{N}}$ be a sequence in $X$ and $x\in X.$ Then the following are equivalent:
\begin{enumerate}
\item[(i)] $\{x_n\}_{n\in \mathbb{N}}$ is $\mathcal{F}$-convergent to $x.$

\item[(ii)] $D(x_n,x)\rightarrow 0$ as $n\rightarrow \infty.$
\end{enumerate}
\end{proposition}

\begin{proposition}  \cite[\, Proposition 4.6.]{JS1} \label{p3}
Let $(X,D)$ be an $\mathcal{F}$-metric space and $\{x_n\}_{n\in \mathbb{N}}$ be a sequence in $X.$ 
Then $$(x,y)\in X\times X, \lim_{n\rightarrow \infty}D(x_n,x)= \lim_{n\rightarrow \infty}D(x_n,y)=0 \Longrightarrow x=y.$$
\end{proposition}

Very recently Som et. all.  \cite{SAL}  proved that this newly defined structure is metrizable by using the definition of metrizability. However, their proof is technical and a bit lengthy. In this short note, we give two alternative proofs of  metrizability  of this structure  using Chittenden's metrization theorem \cite{CD} and metrization theorem due to Niemytski and Wilson (discussed in Theorem \ref{ew}). It may be noted that these proofs are very simple.  Before proceeding to the metrizability result for $\mathcal{F}$-metric spaces, we recall the metrization result due to Chittenden \cite{CD}.

\begin{theorem} \cite{CD} \label{ch}
Let $X$ be a topological space and $F:X \times X\rightarrow [0,\infty)$ be a distance function on $X$. 
If the distance function $F$ satisfies the following conditions:
\begin{enumerate}
\item[(i)] $F(x,y)=0\Longleftrightarrow x=y~~\mbox{for all}~(x,y)\in X \times X$;

\item[(ii)] $F(x,y)=F(y,x),~~\mbox{for all}~(x,y)\in X \times X$;

\item[(iii)](Uniformly regular) For every $\varepsilon>0$ and $x,y,z \in X$ there exists $\phi(\varepsilon)>0$ such that if $F(x,y)<\phi(\varepsilon)$ and $F(y,z)<\phi(\varepsilon)$ then $F(x,z)<\varepsilon,$
\end{enumerate}
then the topological space $X$ is metrizable.
\end{theorem}
Now in the upcoming theorem, we present, by two different approaches, two short proofs of the metrizability of $\mathcal{F}$-metric spaces. 
The first approach is by using Chittenden's metrization theorem, while the second one is by using Niemytski and Wilson's metrization theorem. 

\begin{theorem} \label{FMS}
Let $(X,D)$ be an $\mathcal{F}$-metric space with $(f,\alpha)\in \mathcal{F}\times [0,\infty)$. Then $X$ is metrizable.
\end{theorem}
\begin{proof}

\textbf{Approach I.}

Let $X$ be an $\mathcal{F}$-metric space with $(f,\alpha)\in \mathcal{F}\times [0,\infty)$. By the definition of an $\mathcal{F}$-metric space, the distance function $D:X \times X\rightarrow [0,\infty)$ satisfies the first two conditions of Chittenden's metrization result, i.e,
\begin{enumerate}
\item[(i)] $D(x,y)=0\Longleftrightarrow x=y~~\mbox{for all}~(x,y)\in X \times X$.

\item[(ii)] $D(x,y)=D(y,x)~~\mbox{for all}~(x,y)\in X \times X$.
\end{enumerate}
Now we prove the third condition, i.e., the "uniformly regular" condition. Let $\varepsilon>0$ and $x,y,z \in X.$ If $x=z$, then $D(x,z)=0$. So in this case $\phi(\varepsilon)=c$ where $c$ is any positive real number will serve the purpose. Let $x\neq z.$ Then $D(x,z)>0.$ So by the definition of an $\mathcal{F}$-metric space we have,
\begin{equation}
f(D(x,z))\leq f(D(x,y)+D(y,z))+\alpha.
\label{aa}
\end{equation}
By the $\mathcal{F}_2$ condition, for $(f(\varepsilon)-\alpha)\in \mathbb{R}$ there exists $\delta>0$ such that $0<t<\delta \Longrightarrow f(t)<f(\varepsilon)-\alpha.$ Let us choose $\phi(\varepsilon)=\frac{\delta}{2}.$ If $D(x,y)< \frac{\delta}{2}$ and $D(y,z)<\frac{\delta}{2}$ then $D(x,y)+D(y,z)<\delta.$ So by the equation \ref{aa}, we have 
$$ f(D(x,z))<f(\varepsilon)$$
$$\Longrightarrow D(x,z)<\varepsilon.$$
This shows that the distance function $D$ of an $\mathcal{F}$-metric space satisfies the uniformly regular condition. Consequently,  by Chittenden's metrization result we can conclude that the $\mathcal{F}$-metric space $X$ is metrizable.

\textbf{Approach II.}

In this part we show that any $\mathcal{F}$-metric $D:X \times X\rightarrow [0,\infty)$ satisfies condition (iii-B) of Theorem \ref{ew}. Interesting reader can also check that, the $\mathcal{F}$-metric $D:X \times X\rightarrow [0,\infty)$ satisfies condition (iii-A) of Theorem \ref{ew}, by proceeding similarly as the proof of ``uniformly regular'' condition in Theorem \ref{FMS} under approach I. Let $a\in X$ and $\{a_n\}_{n\in \mathbb{N}}, \{b_n\}_{n\in \mathbb{N}}$ are two sequences in $X$ such that $D(a_n,a)\rightarrow 0$ and  $D(a_n,b_n)\rightarrow 0$ as $n \rightarrow \infty.$ Let $\varepsilon>0.$ By $\mathcal{F}_2$ condition, for $(f(\varepsilon)-\alpha)\in \mathbb{R}$ there exists $\delta>0$ such that $0<t<\delta \Longrightarrow f(t)<f(\varepsilon)-\alpha.$ For $\frac{\delta}{2}>0,$ there exists $k_1,k_2 \in \mathbb{N}$ such that 
$$D(a_n,a)< \frac{\delta}{2}~\forall~n\geq k_1 \mbox{ and } D(a_n,b_n)< \frac{\delta}{2}~\forall~n\geq k_2.$$
Now if $n\geq \mbox{max}\{k_1,k_2\}$ and $a \neq b_n$, then by the definition of an $\mathcal{F}$-metric space, we have
$$f(D(a,b_n))\leq f(D(a,a_n)+D(a_n,b_n))+\alpha$$
$$\Longrightarrow f(D(a,b_n))<f(\varepsilon)\Longrightarrow D(a,b_n)<\varepsilon.$$
This shows that $D(b_n,a)\rightarrow 0$ as $n \rightarrow \infty.$ 

Thus, by the metrization criterion due to Niemytski and Wilson, we can conclude that the $\mathcal{F}$-metric space $X$ is metrizable.
\end{proof}

\begin{remark}
Let us show now that any $\mathcal{F}$-metric $D$ \cite{JS1} satisfies condition (iii-C) of Theorem \ref{ew}. Let $a \in X$ and $k>0.$ Also, assume that $b\in X$ such that $D(a,b)\geq k.$ We have to find $r>0$ corresponding to $a \in X$ and $k>0$ such that for any $c\in X,$ the condition $$D(a,c)+D(b,c)\geq r$$ is satisfied. By $\mathcal{F}_2$ condition, for $(f(k)-\alpha)\in \mathbb{R}$ there exists $r>0$ such that $0<t<r \Longrightarrow f(t)<f(k)-\alpha.$ Let $c\in X.$ Since $D(a,b)\geq k>0,$ so $a \neq b.$ Then by the definition of an $\mathcal{F}$-metric space we have,

$$f(D(a,b))\leq f\Big(D(a,c)+D(c,b)\Big)+\alpha$$
$$\Longrightarrow f\Big(D(a,c)+D(c,b)\Big)+\alpha\geq f(k) ~[\mbox{since}~D(a,b)\geq k ~\mbox{and}~f\in \mathcal{F}]$$
$$\Longrightarrow f\Big(D(a,c)+D(c,b)\Big)\geq f(k)-\alpha$$

$$\Longrightarrow D(a,c)+D(c,b)\geq r ~[\mbox{since}~0<t<r \Longrightarrow f(t)<f(k)-\alpha].$$
\end{remark}

\begin{remark}
From Theorem \ref{FMS} we can conclude that if $(X,D)$ be an $\mathcal{F}$-metric space then there exists a metric $d:X \times X\rightarrow [0,\infty)$ on $X$ such that $X$ is metrizable with respect to the metric $d.$ So, the topological properties of $\mathcal{F}$-metric spaces discussed in Proposition \ref{p1}-\ref{p3} are equivalent to those of the standard metric counterparts.
\end{remark}

\section{ \bf Metrizability of $\theta$-metric spaces}

In 2013, Khojasteh et all. \cite{ks}  introduced the notion of a $\theta$-metric space by using the concept of an $B$-action on 
the set $[0,\infty)\times [0,\infty).$ Before proceeding to the definition of $\theta$-metric space, we recall the definition of an $B$-action 
(see \cite{ks}), as follows.

\begin{definition}\cite[\, Definition 4.]{ks}
Let $\theta:[0,\infty)\times [0,\infty)\rightarrow [0,\infty)$ be a continuous mapping with respect to each variable. 
Let $Im(\theta)=\{\theta(s,t): s,t\geq 0\}.$ Then $\theta$ is called an $B$-action if and only if the following conditions are satisfied :
\begin{enumerate}
\item[(i)] $\theta(0,0)=0$ and $\theta(s,t)=\theta(t,s) ~\mbox{for all}~ s,t \geq 0$;

\item[(ii)] $\theta(x,y)<\theta(s,t)$ if either $x\leq s,~ y<t$ or $x< s,~ y\leq t$;

\item[(iii)] For each $m\in Im(\theta)$ and for each $t\in [0,m],$ there exists $s\in [0,m]$ such that $\theta(s,t)=m;$

\item[(iv)] $\theta(s,0)\leq s ~\mbox{for all}~ s>0$.

\end{enumerate}

\end{definition}
Authors denoted the collection of all such $B$-actions by $Y$. Now, we will  recall (see \cite{ks}) the definition of a $\theta$-metric space, as follows.

\begin{definition}\cite[\, Definition 11.]{ks}
Let $X$ be a non-empty set. A distance function $d: X \times X\rightarrow [0,\infty)$ is said to be a $\theta$-metric on $X$ 
with respect to an $B$-action $\theta\in Y$  if the following conditions are satisfied:
\begin{enumerate}
\item[(i)] $d(x,y)=0\Longleftrightarrow x=y~\mbox{for all}~(x,y)\in X \times X$;

\item[(ii)] $d(x,y)=d(y,x)~$ $\mbox{for all}~ (x,y)\in X \times X$;

\item[(iii)]$d(x,z)\leq \theta(d(x,y),d(y,z))~$ $\mbox{for all}~ x,y,z\in X$.

\end{enumerate}
\end{definition}
The triple $(X,d,\theta)$ is called a $\theta$-metric space. If we take $\theta(s,t)=s+t, ~s,t\geq 0$ then $\theta$-metric space reduce to metric space. In the same paper, Khojasteh et all. \cite{ks} also developed some topological structure induced by the $\theta$-metric and concluded that it is a metrizable topological space. However their proof of metrizability relies on the prior knowledge of the uniformity of an uniform space $X.$ In our paper,  we prove the metrizability of $\theta$-metric spaces by using the well-known Niemytski and Wilson's metrization theorem.

\begin{theorem} \label{theta}
Let $(X,d,\theta)$ be a $\theta$-metric space where $\theta$ is an $B$-action on $[0,\infty)\times [0,\infty).$ Then $X$ is metrizable.
\end{theorem} 

\begin{proof}
Throughout this proof, we will use the standard norm on the set $[0,\infty)\times [0,\infty)$ as $\|(x,y)\|=\sqrt{x^2+y^2}, x,y\geq 0.$ First of all, we show that the $B$-action $\theta$ is continuous at the point $(0,0).$ Suppose that $\{(s_n,t_n)\}_{n\in \mathbb{N}}$ is a sequence in $[0,\infty)\times [0,\infty)$, 
such that $(s_n,t_n)\rightarrow (0,0)$ as $n\rightarrow \infty.$ This implies $s_{n}\rightarrow 0$ and $t_{n}\rightarrow 0$ as $n\rightarrow \infty$ in the standard norm in $[0,\infty)\times [0,\infty).$ Now, as the $B$-action $\theta$ is continuous in both of the variables, we get that $\theta(s_n,t_n)\rightarrow \theta(0,0)=0$ as $n\rightarrow \infty.$ This shows that the $B$-action $\theta$ is continuous at the point $(0,0).$ Now we prove that $X$ is metrizable. 
By the definition of a $\theta$-metric space, the distance function $d:X \times X\rightarrow [0,\infty)$ on $X$ satisfies the first two conditions of 
Niemytski and Wilson's metrization result, i.e,
\begin{enumerate}
\item[(i)] $d(x,y)=0\Longleftrightarrow x=y~~\mbox{for all}~(x,y)\in X \times X$;

\item[(ii)] $d(x,y)=d(y,x)~~\mbox{for all}~(x,y)\in X \times X$.
\end{enumerate}
Now we show that any $\theta$-metric $d:X \times X\rightarrow [0,\infty)$ satisfies the condition (iii-B) and (iii-C) of Theorem \ref{ew}. Interesting reader can also check that, the $\theta$-metric $d:X \times X\rightarrow [0,\infty)$ also satisfies the condition (iii-A) of Theorem \ref{ew}. Let $a\in X$ and $\{a_n\}_{n\in \mathbb{N}}, \{b_n\}_{n\in \mathbb{N}}$ are two sequences in $X$ such that $d(a_n,a)\rightarrow 0$ and $d(a_n,b_n)\rightarrow 0$ as $n\rightarrow \infty.$ We show that $d(b_n,a)\rightarrow 0$ as $n\rightarrow \infty.$ Now $(d(a_n,a),d(a_n,b_n))\rightarrow (0,0)$ as $n\rightarrow \infty$ in the standard norm on $[0,\infty)\times [0,\infty).$ As the $B$-action $\theta$ is continuous at the point $(0,0)$ so $\theta(d(a_n,a),d(a_n,b_n))\rightarrow \theta(0,0)=0$ as $n\rightarrow \infty.$ Now from the definition of $\theta$-metric space we have,
$$d(a,b_n)\leq \theta(d(a_n,a),d(a_n,b_n))$$
$$\Longrightarrow d(a,b_n)\rightarrow 0~\mbox{as}~n\rightarrow \infty.$$
So the $\theta$-metric $d:X \times X\rightarrow [0,\infty)$ satisfies the condition (iii-B) of Theorem \ref{ew}. Now we check for condition (iii-C).
Let $a\in X$ and $k>0.$ Let $b\in X$ such that $d(a,b)\geq k.$ As the $B$-action $\theta$ is continuous at the point $(0,0)$, so for $k>0$ there exists $\delta>0$ such that $$\theta(x,y)<k~\mbox{whenever}~(x,y)\in B\Big((0,0),\delta \Big)\bigcap \Big([0,\infty)\times [0,\infty) \Big).$$
Here $B\Big((0,0),\delta\Big)$ denotes the open ball centered at $(0,0)$ and radius $\delta$ in the standard norm, i.e, $B\Big((0,0),\delta\Big)=\Big\{(x,y)\in \mathbb{R}^{2}: \|(x,y)\|<\delta \Big\}.$ Let $c\in X.$ From the definition of $\theta$-metric space we have
$$d(a,b)\leq \theta(d(a,c),d(c,b))$$
$$\Longrightarrow \theta(d(a,c),d(c,b))\geq k$$
$$\Longrightarrow (d(a,c),d(c,b))\notin B\Big((0,0),\delta \Big)\bigcap \Big([0,\infty)\times [0,\infty) \Big)$$
$$\Longrightarrow d^{2}(a,c)+d^{2}(c,b)\geq \delta^{2}$$
$\mbox{as}~(d(a,c),d(c,b))\in [0,\infty)\times [0,\infty),~\mbox{so,}~(d(a,c),d(c,b))\notin  B\Big((0,0),\delta \Big).$

So either $d(a,c)\geq \frac{\delta}{\sqrt{2}}$ or $d(c,b)\geq \frac{\delta}{\sqrt{2}}.$ So we have $d(a,c)+d(c,b)\geq \frac{\delta}{\sqrt{2}}.$ This shows that the $\theta$-metric on $X$ satisfies condition (iii-C) of Theorem \ref{ew}. Thus, by the metrization criterion due to Niemytski and Wilson, we can conclude that, the $\theta$-metric space $X$ is metrizable.
\end{proof}

%

{\bf Open question.} Can an explicit metric  $d$ separately be constructed with respect to which $b$-metric spaces with coefficient $K>0$ and $\theta$-metric spaces are metrizable ?

\begin{Acknowledgement}
 The Research is funded by the Council of Scientific and Industrial Research (CSIR), Government of India under the Grant Number: $25(0285)/18/EMR-II$. 
We express our deep gratitude to Professor Pratulananda Das for his valuable suggestions during the preparation of the draft. 
\end{Acknowledgement}

\bibliographystyle{plain}

\end{document}